\documentclass{amsart}

\usepackage{amsmath}
\usepackage{amssymb}
\usepackage{datetime}
\usepackage{enumitem}
\usepackage{hyperref}

\newtheorem{theorem}{Theorem}[section]
\newtheorem{lemma}[theorem]{Lemma}
\newtheorem{proposition}[theorem]{Proposition}
\newtheorem{corollary}[theorem]{Corollary}

\newtheorem*{theoremA}{Theorem A}

\theoremstyle{definition}

\theoremstyle{remark}
\newtheorem{remark}[theorem]{Remark}

\numberwithin{equation}{section}

\newcommand{\R}{\mathbb{R}}
\newcommand{\C}{\mathbb{C}}

\newcommand{\Kill}{\mathrm{Kill}}
\newcommand{\Iso}{\mathrm{Iso}}

\newcommand{\HH}{\mathcal{H}}
\newcommand{\GG}{\mathcal{G}}
\newcommand{\UU}{\mathcal{U}}

\newcommand{\WW}{\mathcal{W}}

\newcommand{\FF}{\mathcal{F}}

\newcommand{\g}{\mathfrak{g}}
\newcommand{\h}{\mathfrak{h}}
\newcommand{\s}{\mathfrak{s}}

\newcommand{\so}{\mathfrak{so}}

\newcommand{\sli}{\mathfrak{sl}}
\newcommand{\gl}{\mathfrak{gl}}

\newcommand{\rad}{\mathrm{rad}}

\newcommand{\LA}{\langle}
\newcommand{\RA}{\rangle}
\newcommand{\TG}{\widetilde{G}}
\newcommand{\TM}{\widetilde{M}}

\newcommand{\sS}{\mathrm{S}}

\newcommand{\spi}{\mathfrak{sp}}

\newcommand{\ev}{\mathrm{ev}}

\newcommand{\TSp}{\widetilde{\mathrm{Sp}}}

\begin{document}

\title{Classification of $\big(\TSp(n,\R)\times\TSp(1,\R)\big)$-manifolds}
\author{Gestur \'Olafsson}
\address{Department of Mathematics, Louisiana State University, Baton Rouge, Louisiana}
\email{olafsson@math.lsu.edu}
\author{Eli Roblero-M\'endez}
\address{Department of Mathematics, Louisiana State University, Baton Rouge, Louisiana}
\email{elir@math.lsu.edu}

\keywords{Semisimple Lie groups, rigidity results, pseudo-Riemannian manifolds}
\subjclass{57S20, 53C24, 53C50}

\begin{abstract}
 Let $M$ be an analytic complete finite volume pseudo-Riemannian manifold and $\TSp(n,\R)\times\TSp(1,\R)$
 a connected semisimple Lie group such that its Lie algebra is $\spi(n,\R)\oplus\spi(1,\R)$.
 We characterize the structure of the manifold $M$ assuming that the Lie group $\TSp(n,\R)\times\TSp(1,\R)$
 acts isometrically on $M$ and that its dimension satisfies $3+n(2n+1)<\dim(M)\leq(n+1)(2n+3)$.
\end{abstract}

\maketitle

\section*{Introduction}\label{sect: Intro}
 Let $G$ be a connected non-compact simple Lie group acting isometrically on a connected analytic
 manifold $M$ with a pseudo-Riemannian metric of finite volume. It has been conjectured that
 such actions are rigid, in the sense that restrict the possibilities for $M$. Such conjecture is
 consequence of the program proposed by Robert Zimmer (see **). A principal belief is that such
 action together with other non-trivial assumptions imply that $M$ is, basically, the double coset of a
 semisimple Lie group $H$. Specifically, we have a homomorphism $G\to H$, the existence of a
 compact subgroup $K\subset H$, centralizing the image of $G$, and a lattice $\Gamma\subset H$ such that
 $M$ is isometric to $\Gamma\backslash H/K$.

 Some results have been found in this subject, for example the actions of the Lie groups
 $SO(p,q)$ and $\widetilde U(p,q)$, where $p$ and $q$ are non-negative integer numbers
 (see \cite{OQ},\cite{OQ-Upq}). Note that in the latter case that the Lie group $\widetilde{U}(p,q)$
 is not simple, therefore, there is evidence to suppose that the previous conjecture
 can be true with other Lie groups not necessarily simple.

 In such context we present this research, here we analyze the isometric action of the semisimple
 Lie group $\big(\TSp(n,\R)\times\TSp(1,\R)\big)$ on a connected pseudo-Riemannian manifold of finite volume,
 assuming that both the action and the ma\-ni\-fold are analytic.

 In this paper, for any connected manifold $N$ we denote by $\widetilde{N}$ the simple connected universal
 covering of $N$. Let $G_i$ be a non-compact, connected simple Lie group with Lie algebra $\g_i$, $i=1,2$.
 In this case $G:=G_1\times G_2$ is a semisimple Lie group without compact factors with Lie algebra
 $\g=\g_1\oplus\g_2$. Let $M$ be a connected, finite-volume, pseudo Riemannian manifold which
 admits an analytic and isometric $G$-action with a dense orbit where no factor of $G$ acts trivially.
  As in \cite{OQ}, we prove that for $M$ a weakly irreducible and complete manifold there is a lower bound
 of its dimension given by the dimension of the semisimple Lie group and the properties of the representations
 of its Lie algebra. In other words
 \begin{equation*}
   \dim(M)\geq\dim(G)+m_0(\g_1,\g_2)
 \end{equation*}
 where $m_0(\g_1,\g_2)$ denotes the dimension of the smallest non-trivial representation of both Lie algebras,
 $\g_1$ and $\g_2$, preserving a non-degenerate symmetric bilinear form.

 As a consequence of our research we have the following theorem who main result says that such action,
 together with other conditions on the manifold and the action, imply that $M$ is isomorphic, up to a
 finite covering, to a quotient map of a simple Lie group over a lattice.

 \begin{theoremA}\label{th: A}
   We assume the semisimple Lie group $G=\big(\TSp(n,\R)\times\TSp(1,\R)\big)$, for $n\geq3$, acts
   isometrically with a dense orbit on a connected, finite-volume, complete, pseudo-Riemannian
   manifold $M$, where no factor of $G$ acts trivially.
   Assume that $M$ and the $G$-action on $M$ are both analytic.
   If $M$ is weakly irreducible and satisfies that $\dim(M)=(n+1)(2n+3)$, then there exist:
   \begin{itemize}
     \item a lattice $\Gamma\subset\TSp(n+1,\R)$, and
     \item an analytic finite covering map $\tau:\TSp(n+1,\R)/\Gamma\to M$,
   \end{itemize}
   such that $\tau$ is $\big(\TSp(n,\R)\times\TSp(1,\R)\big)$-equivariant. We can also rescale the metric
   on $M$ along the $\TSp(n,\R)$ and $\TSp(1,\R)$-orbits and the normal bundle to the
   $\big(\TSp(n,\R)\times\TSp(1,\R)\big)$-orbits, such that $\tau$ is a local isometry for the metric on
   $\TSp(n+1,\R)$ given by the Killing form of its Lie algebra.
 \end{theoremA}

 The proof of our principal result is based in the tools developed by Gromov and Zimmer through the study
 of the properties of representation of the Killing vector fields that centralize the action of the semisimple
 group. One of our principal tools is Proposition \ref{prop: conull sub}, who is a generalization of
 a similar result in \cite{OQ} and \cite{Q-TINC}. Such proposition shows the existence of a Lie algebra
 $\g(x)$, isomorphic to $\spi(n,\R)\oplus\spi(1,\R)$, of Killing vector fields vanishing in a point
 $x\in\TM$. Such Lie algebra $\g(x)$ induces a structure of $\g$-module on $T_x\TM$ which allows
 the use of representation theory to analyze the normal bundle to the foliation generated by
 the orbits of the action.
 The $\g$-module structure of $T_x\TM$ is closely related to a structure of $\g$-module of
 $\HH$, the set of Killing vector fields that centralize the action of the group $G$. Such structure
 gives us more tools to understand the properties of $\HH$, which instead gives place to the action of
 another Lie group on $M$. The proof of the existence of the centralizer $\HH$ of the action can be
 found Section \ref{sect: first res}. In Section \ref{sect: Struct} we analyze the properties of $\HH$
 and its relation with the tangent space at some point in $M$. The new action induced by the
 centralizer is an important tool for the proof of Theorem A, which can be found in \ref{sect: Proof Main Th}.
 Meanwhile, in Appendix A we have results about the representations of $\TSp(n,\R)$ and its Lie
 algebra $\spi(n,\R)$, which are used in the previous sections.

\section{First results}\label{sect: first res}
 Let $G$ be a semisimple Lie group as in the introduction. We assume that $G$ acts isometrically
 with a dense orbit on a connected, finite-volume, pseudo-Riemannian manifold $M$ where no factor of
 $G$ acts trivially.  Hence, the $G$-action is locally free (see \cite[Theorem 4.17]{Szaro}) and its orbits
 define a foliation that we denote by $\FF$. We also denote by $\FF_1$ (resp. $\FF_2$) the foliation defined
 by the $G_1$-orbits (resp. $G_2$-orbits). We consider that $M$ and the $G$-action on $M$ are both analytic.

 For $X\in\g$, we denote by $X^*$ the vector field on the manifold $M$ whose one-parameter group of
 diffeomorphism is given by $\big(\exp(tX)_t\big)$ through the action on the manifold.

 For any given pseudo-Riemannian manifold $N$, we will denote by $\Kill(N)$ the globally defined Killing
 vector fields of $N$. We denote by $\Kill_0(N,x)$ the Lie algebra of Killing vector fields that vanish
 at the given point $x$. For a vector space $W$ with a non-degenerate symmetric bilinear form,
 we will denote by $\so(W)$ the Lie algebra of linear maps on $W$ that are skew-symmetric with respect to
 the bilinear form. The next result is an application of the Jacobi identity.

 \begin{lemma}\label{lem: lambd}
   Let $N$ be a pseudo-Riemannian manifold and $x\in N$. Then, the map $\lambda_x:\Kill_0(N,x)\to\so(T_xN)$
   given by $\lambda_x(Z)(v)=[Z,V]_x$, where $V$ is any vector field such that $V_x=v$, is a well-defined
   homomorphism of Lie algebras.
 \end{lemma}

 An immediate consequence of the previous result is its use in the proof of the following proposition which
 is a generalization of Proposition 1.2 in \cite{OQ}.

 \begin{proposition}\label{prop: conull sub}
   Assume that $G$ acts isometrically with a dense orbit on a connected, finite-volume, pseudo-Riemannian
   manifold $M$, where no factor of $G$ acts trivially. Consider the $\TG$-action on $\TM$ lifted from
   the $G$-action on $M$. Assume that $M$ and
   the $G$-action on $M$ are both analytic. Then, there exists a conull subset $\sS\subset\TM$ such that
   for every $x\in\sS$ the following properties are satisfied:
   \begin{itemize}
     \item[$(1)$] There is a homomorphism $\rho:\g=\g_1\oplus\g_2\to\Kill(\TM)$ which is an isomorphism onto its
      image $\rho_x(\g)=\g(x)$.
     \item[$(2)$] $\g(x)\subset\Kill_0(\TM,x)$, i.e. every element of $\g(x)$ vanishes at $x$.
     \item[$(3)$] For every $X,Y\in\g$ we have:
      \begin{equation*}
        [\rho_x(X),Y^*]=[X,Y]^*=-[X^*,Y^*].
      \end{equation*}
      In particular, the elements in $\g(x)$ and their corresponding local flows preserve $\FF$,
      $\FF_1$, $\FF_2$ and $T\FF^\perp$.
     \item[$(4)$] The homomorphism of Lie algebras $\lambda_x\circ\rho_x:\g\to\so(T_x\TM)$ induces a $\g$-module
      structure on $T_x\TM$ for which the subspaces $T_x\FF$, $T_x\FF_1$, $T_x\FF_2$ and $T_x\FF^\perp$
      are $\g$-submodules.
     \item[$(5)$] For every $X_i\in\g_i$, $i=1,2$, we have
      \begin{equation*}
        [\rho_x(X_i),X_j^*]=[X_i,X_j]^*=0 \quad \text{with} \quad i\neq j.
      \end{equation*}
   \end{itemize}
 \end{proposition}
 \begin{proof}
   The proof is similar to the proof of Proposition 1.2 in \cite{Q-TINC}. We only note that the result:
   $\overline{\mathrm{Ad}(G)}^Z$ is the algebraic hull of $M\times GL(\g)$ for the product action, is also
   true for semisimple Lie groups without compact factors (see \cite[Example 3.15]{Witte}).
 \end{proof}

 Let $x\in\sS$ and $u\in T_x\FF_1\cap T_x\FF_2$, then there exists $X_i\in\g_i$, for $i=1,2$,
 such that $(X_1^*)_x=(X_2^*)_x=u$. Let $Y_j\in\g_j$ be, for $j=1,2$,
 by Lemma \ref{lem: lambd} and Proposition \ref{prop: conull sub}(3) we have that
 \begin{equation*}
   \lambda_x(\rho_x(Y_1))(u)=[\rho_x(Y_1),X_2^*]_x=[Y_1,X_2]^*_x=0,
 \end{equation*}
 and
 \begin{equation*}
   \lambda_x(\rho_x(Y_2))(u)=[\rho_x(Y_2),X_1^*]_x=[Y_2,X_1]^*_x=0,
 \end{equation*}
 which imply that $u=0$. Since $x\in\sS$ is arbitrary we conclude that $T_x\FF=T_x\FF_1\oplus T_x\FF_2$, for every $x\in\sS$.

 With the above setup, assume that the $G_i$-orbits are non-degenerate
 with respect to the ambient pseudo-Riemannian metric. In particular, the $\TG_i$-orbits
 on $\TM$ are non-degenerate as well and we have a direct sum decomposition $T\TM=T\FF_i\oplus T\FF_i^\perp$.
 Recall the differential form $\omega^i$ (see \cite{Q-TINC}) given, at every point $x\in\TM$, by
 the composition of the natural projection $T_x\TM\to T_x\FF_i$ and the natural isomorphism $T_x\FF_i\cong\g_i$.
 We also recall the $\g_i$-valued $2$-form given by $\Omega^i=d\omega^i|_{\wedge^2T\FF_i^\perp}$, for $i=1,2$.

 \begin{lemma}[{\cite[Lemma 2.5]{Q-TINC}}]\label{lem: normal bundle}
   Let $G$, $M$, and $\sS$ be as in Proposition \ref{prop: conull sub}. If we assume that the $G_i$-orbits
   are non-degenerate, for $i=1,2$, then:
   \begin{itemize}
     \item[$(1)$] For every $x\in\sS$, the maps $\omega^i_x:T_x\TM\to\g_i$ and
         $\Omega^i_x:\wedge^2T_x\FF_i^\perp\to\g_i$
        are $\g_i$-homomorphism, for the $\g_i$-module structures from Proposition
      \ref{prop: conull sub}.
     \item[$(2)$] The normal bundle $T\FF_i^\perp$ is integrable if and only if $\Omega^i=0$.
   \end{itemize}
 \end{lemma}
 \begin{proof}
   The proof is similar as that of Lemma 2.5 in \cite{Q-TINC}, where the simplicity (or semisimplicity)
   of the group does not play a role.
 \end{proof}

 Next, we relate the metric of $T\FF$ coming from $M$ to suitable metrics on $G$.

 \begin{lemma}[{\cite[Lemma 2.6]{Q-TINC}}]\label{lem: Ad-inv}
   Suppose that the $G$-action on $M$ has a dense orbit and preserves a finite-volume pseudo-Riemannian
   metric. Then, for every $x\in M$ and with respect to the natural isomorphism $T_x\FF\cong\g$, the
   metric of $M$ restricted to $T_x\FF$ defines and $\mathrm{Ad}(G)$-invariant symmetric bilinear form
   on $\g$ independent of the point $x$.
 \end{lemma}
 \begin{proof}
   See the proof at Lemma 2.6 in \cite{Q-TINC}.
 \end{proof}

 We assume, from now on, that $\dim(G_2)<\dim(G_1)$. In this case we have that the non-degeneracy of orbits
 is ensured for low-dimensional manifolds by the next result, which is similar to \cite[Lemma 2.7]{Q-TINC}.

 \begin{lemma}\label{lem: non-dege}
  Assume that $G=G_1\times G_2$ acts isometrically and with a dense orbit on a connected finite-volume
  pseudo-Riemannian manifold $M$. If $\dim(M)<2\dim(G_1)$ and if the $G_2$-orbits are non-degenerate,
  then the bundles $T\FF_1$, $T\FF$ and $T\FF^\perp$ have fibers that are non-degenerate with
  respect to the metric on $M$.
 \end{lemma}
 \begin{proof}
  By Lemma \ref{lem: Ad-inv}, for every $x\in M$, the metric $h$ (on $M$) restricted to $T_x\FF$ corresponds
  to an $\mathrm{Ad}(G)$-invariant form in $\g$. The Kernel of such form is $\g$-invariant,
  therefore, since $T_x\FF_i\simeq\g_i$, $\ker(h_x)$ is either $0$, $T_x\FF_1$, $T_x\FF_2$ or $T_x\FF$.

  Assume that $T_x\FF_1\subseteq\ker(h_x)$ for some $x\in M$. Then, $T_x\FF_1$ lies in the null cone
  of $T_x\FF$ for the metric $h_x$. Hence, for the signature of $M$, which we denote as $(m,n)$, we have
  that $\dim(G_1)=\dim(T_x\FF_1)\leq\min(m,n)$. This implies that $2\dim(G_1)\leq m+n=\dim(M)$, which is
  impossible.

  On the other hand, by hypothesis $\ker(h_x)\neq T_x\FF_2$. Therefore and the previous paragraphs
  we have the desired result.
 \end{proof}

 \begin{remark}\label{rem: orthogonal}
   Let us choose and fix an element $x\in\sS$. Let $X_i\in\g_i$ be, for $i=1,2$, if $Z_1\in\g_1$
   then, by Proposition \ref{prop: conull sub}(2)-(3), $\rho_x(Z_1)\in\Kill_0(\TM,x)$ and
   \begin{eqnarray*}
     \LA(X_1)^*_x,(X_2)^*_x\RA_x &=& \LA\rho_x(Z_1)\cdot(X_1)^*_x,\rho_x(Z_1)\cdot(X_2)^*_x\RA_x \\
      &=& \LA[Z_1,X_1]^*_x,[Z_1,X_2]^*_x\RA_x \\
      &=& \LA[Z_1,X_1]^*_x,0^*_x\RA_x\\
      &=& \LA[Z_1,X_1]^*_x,0\RA_x\\
      &=& 0.
   \end{eqnarray*}
   The previous computation and Lemma \ref{lem: non-dege} imply that $T_x\FF$ is an orthogonal direct sum of
   $T_x\FF_1$ and $T_x\FF_2$. In particular, $T_x\FF_1^\perp$ is the orthogonal direct sum of $T_x\FF_2$ and $T_x\FF^\perp$.
 \end{remark}

 If the $G$-orbits are non-degenerate and the normal bundle to such orbits is integrable, then the
 universal covering space can be split.

 \begin{proposition}\label{prop: split}
   Assume $G_i$ (resp. $G$) acts isometrically
   on a connected, complete, finite-volume, pseudo-Riemannian manifold $M$. If the tangent bundle to the
   orbits $T\FF_i$ (resp. $T\FF$) has non-degenerate fibers and the bundle $T\FF_i^\perp$ (resp. $T\FF^\perp$)
   is integrable, then there is an isometric co\-ve\-ring map $\TG_i\times N\to M$ (resp. $\TG\times N\to M$)
   where the domain has the product metric for a
   bi-invariant metric on $\TG_i$ (resp. $\TG$) and with $N$ a complete pseudo-Riemannian manifold, for $i=1,2$.
 \end{proposition}

 Recall, from the proof of Lemma \ref{lem: normal bundle}, that for $X_1\in\g_1$ and if $u,v\in T_x\FF_1^\perp$
 with $U,V$ sections of $T\FF_1^\perp$ extending them, we have that
 \begin{equation}\label{eq: g_1 mod}
   X_1\cdot\Omega^1_x(u\wedge v)=-\omega^1_x([[\rho_x(X_1),U],V]_x)-\omega^1_x([U,[\rho_x(X_1),V]]_x),
 \end{equation}
 in a similar way, we have the same result for the homomorphism $\Omega^2_x$.

 \begin{lemma}\label{lem: TF non-tr}
   Let $G$, $M$ and $\sS$ be as in Proposition \ref{prop: conull sub}. Assume that $M$ is complete
   and weakly irreducible. Then, for almost every $x\in\sS$ we have that $T_x\FF^\perp$ is a non-trivial
   $\g_i$-module, $i=1,2$.
 \end{lemma}
 \begin{proof}
   Since $\TM$ is a weakly irreducible manifold, by Proposition \ref{prop: split} and Lemma
   \ref{lem: normal bundle}(2) we have that $\Omega^1\neq0$, therefore, since the $2$-form $\Omega^1$ is
   clearly analytic, it vanishes on a proper analytic submanifold subset of $\TM$ of measure
   zero. Hence, $\Omega^1_x\neq0$ for almost every $x\in\sS$.
   Let $x\in\sS$ be an arbitrary but fixed element such that $\Omega^1_x\neq0$. Lemma
   \ref{lem: normal bundle}(1) implies that the map $\Omega^1_x:T_x\FF_1^\perp\to\g_1$ is a
   non-trivial map.

   Let $u,v\in T_x\FF_1^\perp$, by Remark \ref{rem: orthogonal},
   there are $u_2,v_2\in T_x\FF_2$ and $\hat{u},\hat{v}\in T_x\FF^\perp$ such that
   $u=u_2+\hat{u}$ and $v=v_2+\hat{v}$. Let $U,V$ be sections of $T\FF^\perp$ such that
   $U_x=\hat{u}$ and $V_x=\hat{v}$, in the same way, let $X_2,Y_2\in\g_2$ be such that $(X_2)^*_x=u_2$
   and $(Y_2)^*_x=v_2$. Now, let $X_1\in\g_1$, by \eqref{eq: g_1 mod}, we have
   \begin{eqnarray*}
     X_1\cdot\Omega^1_x(u\wedge v) &=& X_1\cdot\Omega^1_x\big((u_2+\hat{u})\wedge(v_2+\hat{v})\big) \\
      &=& X_1\cdot\Omega^1_x(u_2\wedge v_2)+ X_1\cdot\Omega^1_x(u_2\wedge \hat{v})+\\
      &&  X_1\cdot\Omega^1_x(\hat{u}\wedge v_2)+ X_1\cdot\Omega^1_x(\hat{u}\wedge\hat{v})\\
      &=& -\omega^1_x([[\rho_x(X_1),X_2^*],Y^*_2]_x)-\omega^1_x([X_2^*,[\rho_x(X_1),Y_2^*]]_x)+\\
      && -\omega^1_x([[\rho_x(X_1),X_2^*],U]_x)-\omega^1_x([X_2^*,[\rho_x(X_1),U]]_x)+\\
      && -\omega^1_x([[\rho_x(X_1),U],Y^*_2]_x)-\omega^1_x([U,[\rho_x(X_1),Y_2^*]]_x)+\\
      && -\omega^1_x([[\rho_x(X_1),U],V]_x)-\omega^1_x([U,[\rho_x(X_1),V]]_x)\\
      &=& -\omega^1_x([0,Y^*_2]_x)-\omega^1_x([X_2^*,0]_x)+\\
      && -\omega^1_x([0,U]_x)-\omega^1_x([X_2^*,0]_x)+\\
      && -\omega^1_x([0,Y^*_2]_x)-\omega^1_x([U,0]_x)+\\
      && -\omega^1_x([0,V]_x)-\omega^1_x([U,0]_x)\\
      &=& 0.
   \end{eqnarray*}
   As $X_1\in\g_1$ was arbitrary, it follows that
   $\Omega^1_x(T_x\FF_2\wedge T_x\FF^\perp)=0$. On the other hand, because $\Omega^1_x\neq0$,
   we have that $\Omega^1_x(\wedge^2T_x\FF^\perp)\neq0$,  therefore, we have that $T_x\FF^\perp$ is not
   a trivial $\g_1$-module.

   In a similar way, to the previous steps, we can prove that $T_x\FF^\perp$ is not a trivial $\g_2$-module.
 \end{proof}

 For $i=1,2$, let $m(\g_i)$ be the dimension of the smallest non-trivial representation of $\g_i$ that admits an
 invariant non-degenerate symmetric bilinear form.
 Since $\g=\g_1\oplus\g_2$, we define $m(\g)=m(\g_1,\g_2)$
 the dimension of the smallest non-trivial representation of both $\g_1$ and $\g_2$ that admits an
 invariant non-degenerate symmetric bilinear form. If we assume that there is a non-trivial homomorphism
 $\g_2\to\g_1$ then $m(\g)\leq m(\g_1)$, even more we have that $m(\g)=m(\g_1)$.

 From now on we assume the existence of an injective homomorphism $\g_2\to\g_1$.
 A consequence of the previous result is the obtention of a lower bound on the dimension of $M$.

 \begin{proposition}\label{prop: dim}
   Let $M$ be a connected analytic pseudo-Riemannian manifold. Suppose that $M$ is
   complete, weakly irreducible, has finite-volume and admits an analytic isometric, non-transitive
   $G$-action with a dense orbit and such that no factor acts trivially. We also assume that the
   $G_2$-orbits are non-degenerate. If $m(\g)+\dim(G_2)\leq\dim(G_1)$ then
   \begin{equation*}
     \dim(M)\geq\dim(G)+m(\g).
   \end{equation*}
 \end{proposition}
 \begin{proof}
   Suppose that $\dim(M)<\dim(G)+m(\g)$. Since $m(\g)+\dim(G_2)\leq\dim(G_1)$ then $\dim(M)<2\dim(G_1)$
   and, by Lemma \ref{lem: non-dege}, the bundle $T\FF^\perp$ has non-degenerate fibers with
   dimension $<m(\g)$. Hence, Lemma \ref{lem: TF non-tr} and the definition of $m(\g)$ imply that
   $T_x\FF^\perp$ is a trivial $\g_1$-module for the structure defined by Proposition
   \ref{prop: conull sub}(4), hence Proposition \ref{prop: split} contradicts the irreducibility of $M$.
 \end{proof}

 For a $G$-action as in Proposition \ref{prop: conull sub}, consider $\TM$ endowed with the $\TG$-action
 obtained by lifting the $G$-action on $M$. With such setup, let us denote by $\HH$ the Lie subalgebra
 of $\Kill(\TM)$ consisting of the fields that centralize the $\TG$-action on $\TM$.

 \begin{lemma}\label{lem: submodule}
   Let $\sS$ be as in Proposition \ref{prop: conull sub}. Then, for every $x\in\sS$ and for $\rho_x$
   as in Proposition \ref{prop: conull sub}, the map $\widehat{\rho}_x:\g\to\Kill(\TM)$ given by:
   \begin{equation*}
     \widehat{\rho}(X)=\rho_x(X)+X^*
   \end{equation*}
   is an injective homomorphism of Lie algebras whose image $\GG(x)$ lies in $\HH$. In particular,
   $\widehat{\rho}_x$ induces on $\HH$ a $\g$-module structure such that $\GG(x)$ is a submodule
   isomorphic to $\g$.
 \end{lemma}
 \begin{proof}
   The identity in Proposition \ref{prop: conull sub}(3) implies that the image of $\widehat{\rho}_x$
   lies in $\HH$.

   Let $X_i,Y_i\in\g_i$ be, for $i=1,2$, then applying Proposition \ref{prop: conull sub}(3) we have
   \begin{eqnarray*}
     [\widehat{\rho}_x(X_1+X_2),\widehat{\rho}_x(Y_1+Y_2)] &=&
       [\rho_x(X_1+X_2)+(X_1+X_2)^*,\\
        && \negmedspace {}\rho_x(Y_1+Y_2)+(Y_1+Y_2)^*] \\
      &=& [\rho_x(X_1)+X_1^*+\rho_x(X_2)+X_2^*,\\
       && \negmedspace {}\rho_x(Y_1)+Y_1^*+\rho_x(Y_2)+Y_2^*]\\
      &=& \sum_{i,j=1}^{2} [\rho_x(X_i)+X_i^*,\rho_x(Y_j)+Y_j^*]\\
      &=& \sum_{i,j=1}^{2} \big([\rho_x(X_i),\rho_x(Y_j)]+[\rho_x(X_i),Y_j^*]+\\
       && \negmedspace {}[X_i^*,\rho_x(Y_j)]+[X_i^*,Y_j^*]\big) \\
      &=& \sum_{i,j=1}^{2} \big(\rho_x([X_i,Y_j])+[X_i,Y_j]^*+\\
       && \negmedspace {}[X_i,Y_j]^*-[X_i,Y_j]^*\big)\\
      &=& \sum_{i,j=1}^{2} \big(\rho_x([X_i,Y_j])+[X_i,Y_j]^*\big)\\
      &=& \big(\rho_x([X_1+X_2,Y_1+Y_2])+[X_1+X_2,Y_1+Y_2]^*\big)\\
      &=& \widehat{\rho}_x([X_1+X_2,Y_1+Y_2]).
   \end{eqnarray*}
   If $X\in\g$ is an element which satisfies that $\widehat{\rho}_x(X)=0$ then
   $X^*=\rho_x(X)+X^*=0$, which, by locally freeness, implies $X=0$.
 \end{proof}

 Following, we relate the $\g$-module structure associated to $\HH$ and to $T_x\TM$, respectively.

 \begin{lemma}\label{lem: eval}
   Let $\sS$ be as in Proposition \ref{prop: conull sub}. Consider $T_x\TM$ and $\HH$ endowed with the
   $\g$-module structure given by Proposition \ref{prop: conull sub} and Lemma \ref{lem: submodule},
   respectively. Then, for every $x\in\sS$, the evaluation map
   \begin{equation*}
     \mathrm{ev}_x:\HH\to T_x\TM, \qquad Y\mapsto Y_x
   \end{equation*}
   is a homomorphism of $\g$-modules that satisfies $\mathrm{ev}_x(\GG(x))=T_x\FF$. Furthermore, for
   almost every $x\in\sS$ we have $\mathrm{ev}_x(\HH)=T_x\TM$.
 \end{lemma}
 \begin{proof}
   For every $x\in\sS$, let $Y\in\HH$ and $X\in\g$ be given, then
   \begin{eqnarray*}
     \ev_x(X\cdot Y) &=& [\widehat{\rho}_x(X),Y]_x \\
      &=& [\rho_x(X)+X^*,Y]_x \\
      &=& [\rho_x(X),Y]_x+[X^*,Y]_x \\
      &=& [\rho_x(X),Y]_x \\
      &=& \rho_x(X)\cdot Y_x \\
      &=& X\cdot \ev_x(Y)
   \end{eqnarray*}
   where we have used the definition of $\g$-module structures involved and properties of the map $\lambda_x$
   (Lemma \ref{lem: lambd}).
   The last claim follows by an adaptation of the proof of Lemma 4.1 of \cite{Zim-Entropy} and
   Theorem 3.1 of \cite{Nevo}, which establish the transitivity of $\HH$ on an open conull dense.
 \end{proof}

 On a complete manifold every Lie algebra of Killing vector fields can be realized from an isometric
 right action, this is the result of the following Lemma which appears as Lemma 1.11 in \cite{OQ}.

 \begin{lemma}[{\cite[Lemma 1.11]{OQ}}]\label{lem: righ-act}
   Let $N$ be a complete pseudo-Riemannian manifold and $H$ a simply connected Lie group with Lie algebra
   $\h$. If $\psi:\h\to\Kill(N)$ is a homomorphism of Lie algebras, then there exists an isometric right
   $H$-action $N\times H\to N$ such that $\psi(X)=X^*$, for every $X\in\h$. Furthermore, if $N$ is analytic,
   then the $H$-action is analytic as well.
 \end{lemma}

\section{Structure of the centralizer}\label{sect: Struct}

 In this section we assume the case $G=\big(\TSp(n,\R)\times\TSp(1,\R)\big)$ which acts analytical and
 isometrically on a connected, analytic, finite-volume, complete, pseudo-Riemannian manifold $M$ with a
 dense orbit, such that no factors of $G$ acts trivially. Therefore, the results of Section \ref{sect: first res}
 can apply to this case. We also assume that $\dim(M)\leq\dim\big(\TSp(n,\R)\times\TSp(1,\R)\big)+4n=2n^2+5n+3$,
 for $n\geq3$.

 Given the assumptions in the previous paragraph, by Lemma \ref{lem: non-dege}, we have the direct sum
 $TM=T\FF\oplus T\FF^\perp$. Here, we also assume that the manifold $M$ is weakly irreducible.

 \begin{lemma}\label{lem: R(2n,2n)}
   Let $x\in\sS$. Consider $T_x\FF^\perp$ endowed with the
   $\big(\spi(n,\R)\oplus\spi(1,\R)\big)$-module structure given by Proposition \ref{prop: conull sub}(4). Then,
   for almost every $x\in\sS$, the $\big(\spi(n,\R)\oplus\spi(1,\R)\big)$-module $T_x\FF^\perp$ is isomorphic to
   $\R^{2n,2n}$. In particular, $\so(T_x\FF^\perp)$ is isomorphic to $\so(2n,2n)$ as a Lie algebra and as
   an $\big(\spi(n,\R)\oplus\spi(1,\R)\big)$-module.
 \end{lemma}
 \begin{proof}
   Since we are assuming that $M$ is weakly irreducible, by Lemma \ref{lem: TF non-tr}, we have that
   for almost every $x\in\sS$, $T_x\FF^\perp$ is a non-trivial $\spi(n,\R)$-module (respectively
   $\spi(1,\R)$-module), therefore it is a non-trivial $\big(\spi(n,\R)\oplus\spi(1,\R)\big)$-module.

   By Proposition \ref{prop: conull sub}(4) and Lemma \ref{lem: non-dege} we have that the map
   $\lambda_x\circ\rho_x$ induces a non-trivial homomorphism of Lie algebras from
   $\spi(n,\R)$ (resp. $\spi(1,\R)$) into $\so(T_x\FF^\perp)$. Since $\spi(n,\R)$ (resp. $\spi(1,\R)$) is
   a simple Lie algebra we have that such homomorphism is injective.

   Let $X\in\spi(n,\R)$ and $Y\in\spi(1,\R)$ be, if $u\in T_x\FF^\perp$ we chose $U$ a vector field on $\TM$ such
   that $U_x=u$, then
   \begin{eqnarray*}
     \big((\lambda_x\circ\rho_x)(X)\circ(\lambda_x\circ\rho_x)(Y)\big)(u) &=& \big((\lambda_x\circ\rho_x)(X)\circ
      (\lambda_x\circ\rho_x)(Y)\big)(U_x) \\
      &=& (\lambda_x\circ\rho_x)(X)\big(\lambda_x(\rho_x(Y))\big)(U_x) \\
      &=& (\lambda_x\circ\rho_x)(X)([\rho_x(Y),U]_x) \\
      &=& (\lambda_x(\rho_x(X))([\rho_x(Y),U]_x) \\
      &=& [\rho_x(X),[\rho_x(Y),U]]_x\\
      &=& [\rho_x(Y),[\rho_x(X),U]]_x+[[\rho_x(X),\rho_x(Y)],U]_x\\
      &=& [\rho_x(Y),[\rho_x(X),U]]_x+[\rho_x([X,Y]),U]_x\\
      &=& [\rho_x(Y),[\rho_x(X),U]]_x\\
      &=& \big((\lambda_x\circ\rho_x)(Y)\circ(\lambda_x\circ\rho_x)(X)\big)(u)
   \end{eqnarray*}
   therefore, we have that $(\lambda_x\circ\rho_x)(\spi(n,\R))$ and $(\lambda_x\circ\rho_x)(\spi(1,\R))$
   commute each other in $\so(T_x\FF^\perp)$. Hence, the map $\lambda_x\circ\rho_x$ induces an injective
   homomorphism of Lie algebras from $\big(\spi(n,\R)\oplus\spi(1,\R)\big)$ into $\so(T_x\FF^\perp)$.

   Since $T_x\FF^\perp$ is a non-trivial $\spi(n,\R)$-module preserving a non-degenerate symmetric bilinear
   form then, by Lemma \ref{lem: dim rep}, we have that $T_x\FF^\perp$ is isomorphic to
   $\R^{2n}\oplus\R^{2n}$ as $\spi(n,\R)$-module. Let $\{e_1,e_2,\ldots,e_n,e_{n+1},e_{n+2},\ldots,e_{n+n}\}$ be
   the canonical base of $\R^{2n}$. Let $i,j\in\{1,2,\ldots,n\}$ be such that $i\neq j$, by the representation
   of $\spi(n,\R)$ on $\R^{2n}$ we can find $A_{i,j}\in\spi(n,\R)$ such that $A_{i,j}(e_i)=e_i$ and
   $A_{i,j}(e_j)=e_j$, therefore if $\LA\cdot,\cdot\RA$ is a non-degenerate symmetric bilinear form on
   $\R^{2n}\oplus\R^{2n}$ preserved by $\spi(n,\R)$ then
   \begin{equation*}
     0=\LA A_{i,j}(e_i),e_j\RA+\LA e_i, A_{i,j}(e_j)\RA=\LA e_i,e_j\RA+\LA e_i,e_j\RA=2\LA e_i,e_j\RA.
   \end{equation*}
   Hence, a subspace of dimension $2n$ is contained in the nullcone of $\LA\cdot,\cdot\RA$ therefore we have %proved
   that the signature of $\LA\cdot,\cdot\RA$ is $(2n,2n)$.
 \end{proof}

 The results in the previous lemma has an immediate consequence in the proof of the following lemma.

 \begin{lemma}\label{lem: H as spn}
   Let $\sS$ be as in Proposition \ref{prop: conull sub}. Then, for almost every $x\in\sS$ and for the
   $\big(\spi(n,\R)\oplus\spi(1,\R)\big)$-module structure on $\HH$ from Lemma \ref{lem: submodule} there
   is a decomposition into $\big(\spi(n,\R)\oplus\spi(1,\R)\big)$-submodules
   $\HH=\GG(x)\oplus\HH_0(x)\oplus\WW(x)$ such that
   \begin{itemize}
     \item[$(1)$] $\GG(x)=\widehat\rho_x\big(\spi(n,\R)\oplus\spi(1,\R)\big)$ is a Lie subalgebra of $\HH$
      isomorphic to $\big(\spi(n,\R)\oplus\spi(1,\R)\big)$ and $\ev_x(\GG(x))=T_x\FF$.
     \item[$(2)$] $\HH_0(x)=\ker(\ev_x)$, is a Lie subalgebra of $\HH$ isomorphic to a subalgebra
      of $\so(2n,2n)$. Even more, such isomorphism is an isomorphism of
      $\big(\spi(n,\R)\oplus\spi(1,\R)\big)$-modules.
     \item[$(3)$] $\ev_x(\WW(x))=T_x\FF^\perp$ and is isomorphic to $\R^{(2n,2n)}$ as
      $\big(\spi(n,\R)\oplus\spi(1,\R)\big)$-module.
   \end{itemize}
   Here, the evaluation map $\ev_x$ defines an isomorphism  of $\big(\spi(n,\R)\oplus\spi(1,\R)\big)$-modules
   $\GG(x)\oplus\WW(x)\to T_x\TM=T_x\FF\oplus T_x\FF^\perp$ preserving the summands in that order.
 \end{lemma}
 \begin{proof}
   Let $x\in\sS$ which satisfies Lemma \ref{lem: eval} and Lemma \ref{lem: R(2n,2n)}.
   Recall, by Lemma \ref{lem: submodule}, that
   $\GG(x)=\widehat{\rho}_x\big(\spi(n,\R)\oplus\spi(1,\R)\big)$ is a Lie subalgebra contained in $\HH$ and
   isomorphic to $\spi(n,\R)\oplus\spi(1,\R)$.

   Define $\HH_0(x)=\ker(\ev_x)$. By Lemma \ref{lem: eval}, we have that $\HH_0(x)$ is an
   $\big(\spi(n,\R)\oplus\spi(1,\R)\big)$-module of $\HH$. Since, $\HH_0(x)=\HH(x)\cap\Kill_0(\TM,x)$
   it follows that it is a subalgebra.

   Let $Z\in\GG(x)\cap\HH_0(x)$ be, then there is $Y\in\g$ such that $Z=\widehat\rho_x(Y)=\rho_x(Y)+Y^*$.
   The condition $Z\in\HH_0(x)$ implies $0=Z_x$. That is, $Y^*_x=(\rho_x(Y)+Y^*)_x=0$, hence, $Y=0$
   and we have that $\GG(x)\cap\HH_0(x)=\{0\}$. Therefore, by Lemma \ref{lem: eval}, there is a
   subspace $\WW_0(x)$ complementary to $\GG(x)\oplus\HH_0(x)$ in $\HH$.

   Since we have an isomorphism from $\GG(x)\oplus\WW_0(x)$ onto $T_x\TM$ via the evaluation map, we
   choose $\WW(x)$ as the inverse image of $T_x\FF^\perp$ under this isomorphism. Therefore, we obtain
   the desired composition of $\HH$ into $\big(\spi(n,\R)\oplus\spi(1,\R)\big)$-modules.

   Let $\Kill_0(\TM,x,\FF)$ be the Lie algebra of Killing vector fields on $\TM$ which preserves the
   foliation $\FF$ and vanish at $x\in\TM$. Note that every vector field in $\Kill_0(\TM,x,\FF)$ leaves
   invariant the normal bundle, thence the map $\lambda_x$ induces the following homomorphism of Lie
   algebras:
   \begin{equation*}
     \lambda_x^\perp:\Kill_0(\TM,x,\FF)\to\so(T_x\FF^\perp), \qquad X\mapsto\lambda_x(X)|_{T_x\FF^\perp}.
   \end{equation*}
   Observe that both $\rho_x\big(\spi(n,\R)\oplus\spi(1,\R)\big)$ and $\HH_0(x)$ lie inside of
   $\Kill_0(\TM,x,\FF)$.

   Claim 1: \emph{$\lambda_x^\perp$ is injective when it is restricted to
   $\big(\spi(n,\R)\oplus\spi(1,\R)\big)(x)$.} By our choice of the element $x\in\sS$, the proof of
   this claim is similar to the proof of Lemma \ref{lem: R(2n,2n)}.

   Claim 2: \emph{$\lambda_x^\perp$ is injective when it is restricted to $\HH_0(x)$.} Recall that
   pseudo-Riemannian metrics are $1$-rigid (see \cite{GCT-CQ}). Therefore, a Killing vector field is completely
   determined by its $1$-jet at $x$. Let $Z\in\HH_0(x)$, then $\ev_x(Z)=Z_x=0$, so it is determined by its
   values $[Z,V]_x$ for $V$ vector field on a neighborhood of $x$. Since $Z\in\HH$ then $[Z,X^*]_x=0$ for
   all $X\in\big(\spi(n,\R)\oplus\spi(1,\R)\big)$ in this way $[Z,V]_x=0$ when $V_x\in T_x\FF$. Hence, if
   $[Z,X]_x=0$ when $V_x\in T_x\FF^\perp$ this implies that $Z=0$. Thence, we have that $\lambda_x^\perp$
   is injective when it is restricted to $\HH_0(x)$.

   Let $X\in\big(\spi(n,\R)\oplus\spi(1,\R)\big)$ and $Y\in\HH_0(x)$, then
   \begin{eqnarray*}
     \lambda_x^\perp(X\cdot Y) &=& \lambda_x^\perp([\widehat\rho_x(X),Y])=\lambda_x^\perp([\rho_x(X)+X^*,Y]) \\
      &=& \lambda_x^\perp([\rho_x(X)+X^*,Y])=[\lambda_x^\perp(\rho(X)),\lambda_x^\perp(Y)] \\
      &=& X\cdot\lambda_x^\perp(Y),
   \end{eqnarray*}
   which shows that the map $\lambda_x^\perp$ restricted to $\HH_0(x)$ is a homomorphism of
   $\big(\spi(n,\R)\oplus\spi(1,\R)\big)$-modules.
 \end{proof}

 It follows from the proof of Lemma \ref{lem: H as spn} that $\GG(x)\oplus\HH_0(x)$ is a Lie subalgebra which
 contains $\HH_0(x)$ as an ideal. We also have that $T_x\FF$ is a trivial $\HH_0(x)$-module, therefore,
 by Lemma \ref{lem: non-dege}, $T_x\FF^\perp$ is a $\HH_0(x)$-module which is non-trivial if and only if
 $\HH_0(x)$ is non-trivial.

 \begin{remark}\label{rem: rho tilde}
   Let $x\in\sS$ as in the previous lemma, if $X\in\g$ and $u\in T_x\TM$ then, by Lemma \ref{lem: eval},
   there exists $U\in\HH$ such that $U_x=u$, hence
   \begin{equation*}
     X\cdot u=[\rho_x(X),U]_x=[\rho_x(X)+X^*,U]_x=[\widehat{\rho}_x(X),U]_x.
   \end{equation*}
   In particular, we can define an action of $\GG(x)$ on $T_x\FF^\perp$ as following
   \begin{equation}\label{eq: G on F-perp}
     \widehat{\rho}_x(X)\cdot u:=[\widehat{\rho}_x(X),U]_x=[\rho_x(X),U]_x.
   \end{equation}
 \end{remark}

 Let $x\in\sS$ be as in Lemma \ref{lem: H as spn}. If $X_1\in\g_1$, $u\in T_x\FF_2$ and $v\in T_x\FF^\perp$
 then there exist $X_2\in\g_2$ and $V\in\WW(x)$ such that $(X_2^*)_x=\widehat\rho_x(X_2)=u$ and $V_x=v$.
 By the proof of Lemma %\ref{lem: normal bundle}
 \ref{lem: TF non-tr} we have that $X_1\cdot\Omega^1_x(u\wedge v)=\Omega^1_x(X_1\cdot(u\wedge v))=0$,
 since $[\widehat\rho_x(X),V]\in\WW(x)$ for every $X\in\g_1\oplus\g_2$. Therefore, by the weak irreducibility
 of $\TM$, we have that $\Omega^1_x=d\omega^1_x|_{\wedge^2T_x\FF^\perp}\neq0$. In a similar way we have that
 $\Omega^2_x=d\omega^2_x|_{\wedge^2T_x\FF^\perp}\neq0$. With the previous result we have that the
 $2$-form $\Omega_x=d\omega_x|_{\wedge^2T_x\FF^\perp}:\wedge^2T_x\FF^\perp\to\g_1\oplus\g_2$ is surjective.

 On the other hand, Lemma $A.5$ in \cite{OQ-Upq} shows the existence of an isomorphism of
 $\so(T_x\FF^\perp)$-modules $\varphi_x:\wedge^2T_x\FF^\perp\to\so(T_x\FF^\perp)$. Therefore, we will denote
 the linear map given by the composition $\Omega_x\circ\varphi^{-1}_x:\so(T_x\FF^\perp)\to\g_1\oplus\g_2$
 with the same symbol $\Omega_x$.

 \begin{proposition}\label{prop: Ker-Omega}
   For $G$ and $M$ as in Proposition \ref{prop: conull sub}. If $T\TM=T\FF\oplus T\FF^\perp$ then
   for almost every $x\in\sS$, the following properties hold:
   \begin{enumerate}
     \item For every $X\in\g_1\oplus\g_2$ and $Y\in\mathfrak{X}(\TM)$ we have
      \begin{equation*}
        \omega_x([\rho_x(X),Y]_x)=[X,\omega_x(Y)].
      \end{equation*}
     \item The linear map $\Omega_x:\wedge^2T_x\FF^\perp\to\g_1\oplus\g_2$ intertwines the homomorphism of Lie
      algebras $\widetilde\rho_x:\g_1\oplus\g_2\to\GG(x)$ for the actions of $\g_1\oplus\g_2$ on
      $\g_1\oplus\g_2$ and of $\GG(x)$ on $T_x\FF^\perp$ via \eqref{eq: G on F-perp}. More precisely
      we have
      \begin{equation*}
        [X,\Omega_x(u\wedge v)]=\Omega_x[\widehat{\rho}_x(X)(u\wedge v)]
      \end{equation*}
      for every $X\in\g_1\oplus\g_2$ and $u,v\in T_x\FF^\perp$.
     \item The linear map $\Omega_x:\so(T_x\FF^\perp)\to\g_1\oplus\g_2$ is $\HH_0(x)$ via $\lambda_x^\perp$.
      More precisely, we have
      \begin{equation*}
        [\lambda_x^\perp(\HH_0(x)),\so(T_x\FF^\perp)]\subset\ker(\Omega_x).
      \end{equation*}
   \end{enumerate}
 \end{proposition}
 \begin{proof}
   The proof is similar to the proof of the Proposition 3.10 in \cite{OQ-Upq}. In that proof the authors prove
   this is true for $\mathcal{D}(\g)=[\g,\g]$, where $\g$ is a simple Lie algebra.

   Those arguments are the same in our case $\g=\g_1\oplus\g_2$, where $\g_1$ and $\g_2$ are simple Lie algebras
   and, therefore, $[\g,\g]=\g$.
 \end{proof}

 \begin{remark}\label{rem: H=0}
   By Lemma \ref{lem: H as spn} we have that $\so(T_x\FF^\perp)\simeq\so(2n,2n)$.
   On the other hand, by the decomposition of $\so(2n,2n)$ as a direct sum of irreducible
   $\big(\spi(n,\R)\oplus\spi(1,\R)\big)$-modules and the results of Proposition \ref{prop: Ker-Omega} we have
   that $\lambda_x^\perp(\HH_0(x))=0$ and, therefore, $\HH_0(x)=0$.

  Let $x\in\sS$ as in Lemma \ref{lem: H as spn}, by the previous remark, Lemma \ref{lem: eval} and
  \eqref{eq: G on F-perp} the evaluation map
  \begin{equation}\label{eq: eval}
    \ev_x:\HH=\GG_1(x)\oplus\GG_2(x)\oplus\WW(x)\to T_x\FF_1\oplus T_x\FF_2\oplus T_x\FF^\perp=T_x\TM
  \end{equation}
  is an isomorphism of $(\g_1\oplus\g_2)$-modules.
 \end{remark}

 \begin{lemma}\label{lem: H simple}
   For the $\big(\TSp(n,\R)\times \TSp(1,\R)\big)$-action on $M$ as in Proposition \ref{prop: conull sub}, assume
   that $n\geq3$. For almost every $x\in\sS$ we have that $\HH$ is a simple Lie algebra isomorphic
   to $\spi(n+1,\R)$.

 \end{lemma}
 \begin{proof}
   We choose an element $x\in\sS$ which satisfies Lemma \ref{lem: H as spn} and Proposition \ref{prop: Ker-Omega}.
   By Remark \ref{rem: H=0} we have that
   $\HH=\GG_1(x)\oplus\GG_2(2)\oplus\WW(x)$ is a Lie algebra where $\GG(x)=\GG_1(x)\oplus\GG_2(x)$,
   $\GG_1(x)=\widehat{\rho}_x(\spi(n,\R))$ and $\GG_2(x)=\widehat{\rho}_x(\spi(1,\R))$.

   Since $\GG(x)=\GG_1(x)\oplus\GG_2(2)$ is a semisimple Lie subalgebra of $\HH$ isomorphic to
   $\spi(n,\R)\oplus\spi(1,\R)$. Let $\s$ be a Levi factor of $\HH$ which contains to $\GG(x)$.

   Recall that the structure of $\HH$ as an $\big(\spi(n,\R)\oplus\spi(1,\R)\big)$-module is given by the
   subalgebra $\GG(x)$ and the Lie brackets in $\HH$. Hence, since $\GG(x)\subset\s$ we have that $\s$
   has a decomposition into $\big(\spi(n,\R)\oplus\spi(1,\R)\big)$-modules. Let $\UU$ be an
   $\big(\spi(n,\R)\oplus\spi(1,\R)\big)$-submodule of $\HH$ such that
   $\s=\GG_1(x)\oplus\GG_2(x)\oplus\UU$. Therefore, we have a decomposition of $\HH$ into
   $\big(\spi(n,\R)\oplus\spi(1,\R)\big)$-module as
   \begin{eqnarray*}
     \HH&=&\s\oplus\rad(\HH)\\
      &=&\GG(x)\oplus\UU\oplus\rad(\HH)\\
      &=&\GG_1(x)\oplus\GG_2(x)\oplus\UU\oplus\rad(\HH)
   \end{eqnarray*}
   that we compare with the decomposition of $\HH$ into $\big(\spi(n,\R)\oplus\spi(1,\R)\big)$-module given
   by Lemma \ref{lem: H as spn} and Remark \ref{rem: rho tilde}, which is
   \begin{equation*}
     \HH=\GG_1(x)\oplus\GG_2(x)\oplus\WW(x).
   \end{equation*}

   Comparing the previous decomposition of $\HH$ as $\big(\spi(n,\R)\oplus\spi(1,\R)\big)$-modules we
   have two possibilities:
   \begin{itemize}
     \item[$(1)$] $\rad(\HH)=\WW(x)$ and $\s=\GG_1(x)\oplus\GG_2(x)$.
     \item[$(2)$] $\HH$ is a semisimple Lie algebra.
   \end{itemize}

   Let us consider the case $\rad(\HH)=\WW(x)$ and $\s=\GG_1(x)\oplus\GG_2(x)$.

   Since $\rad(\HH)$ is an ideal of $\HH$ we have that $[\WW(x),\WW(x)]\subset\WW(x)$, therefore
   $\Omega^i_x([\WW(x),\WW(x)])=0$, for $i=1,2$. By the proof of Lemma \ref{lem: TF non-tr} and \eqref{eq: eval}
   we have that this is not possible. Then case (1) is not possible.

   Now assume that $\HH$ is a semisimple Lie algebra.

   By properties of \eqref{eq: eval} and the action of $\big(\spi(n,\R)\oplus\spi(1,\R)\big)$ into $T_x\FF^\perp$
   we have that $[\GG_i(x),\WW(x)]=\WW(x)$, for $i=1,2$. Hence, if $\h$ is an
   ideal of $\HH$ containing $\GG_i(x)$ then $\h$ must contain $\WW(x)$. Therefore, we have that
   $\HH$ is a simple Lie algebra.

   Since $\WW(x)\simeq T_x\FF^\perp$ as $\big(\spi(n,\R)\oplus\spi(1,\R)\big)$-modules we have then that
   $\wedge^2\WW(x)$ is isomorphic to $\wedge^2T_x\FF^\perp$ as $\big(\spi(n,\R)\oplus\spi(1,\R)\big)$-modules.
   On the other hand, by Lemma A.5 in \cite{OQ-Upq} there is an isomorphism
   $\varphi:\wedge^2T_x\FF^\perp\to\so(T_x\FF^\perp)$ of $\so(T_x\FF^\perp)$-modules.

   From the decomposition of $\so(T_x\FF^\perp)$ (which is isomorphic to $\so(2n,2n)$) as a direct
   sum of $\big(\spi(n,\R)\oplus\spi(1,\R)\big)$-modules (see the proof of Lemma \ref{lem: inclusion}) and
   the fact that $[\WW(x),\WW(x)]$ has non-zero projection on $\GG_1(x)$ and $\GG_2(x)$ we have then
   that $[\WW(x),\WW(x)]=\GG_1(x)\oplus\GG_2(x)$. It follows in particular that
   $\big(\HH,\GG_1(x)\oplus\GG_2(x)\big)$ is a symmetric pair. Therefore,
   by Table II in \cite{Berger}, $\HH$ is isomorphic to $\spi(n+1,\R)$.
 \end{proof}

\section{Proof of the Main Theorems}\label{sect: Proof Main Th}

 In this section we assume the case $G=\big(\TSp(n,\R)\times\TSp(1,\R)\big)$, for $n\geq3$, which acts
 analytical and isometrically on a connected, analytic, finite-volume, complete, pseudo-Riemannian
 manifold $M$ with a dense orbit, such that no factors of $G$ acts trivially. Therefore, the results of
 Section \ref{sect: first res} can apply to this case. We also assume that
 $\dim(M)\leq\dim\big(\TSp(n,\R)\times\TSp(1,\R)\big)+4n$

 Given the assumptions in the previous paragraph, by Lemma \ref{lem: non-dege}, we have the direct sum
 $TM=T\FF\oplus T\FF^\perp$. Here, we also assume that the manifold $M$ is weakly irreducible.

 By results in Section \ref{sect: Struct} we have the existence of a conull subset of $\TM$ which,
 we denote with the same letter $\sS$, such that every element $x\in\sS$ satisfies Lemmas
 \ref{lem: H as spn} and \ref{lem: H simple}. From now on we assume $x_0\in\sS$.

 \begin{lemma}\label{lem: isomorph}
   There is an isomorphism
   \begin{equation*}
     \psi:\spi(n+1,\R)=\spi(n,\R)\oplus\spi(1,\R)\oplus\R^{2n,2n}\to\GG_1(x_0)\oplus\GG_2(x_0)\oplus\WW(x)=\HH
   \end{equation*}
   of Lie algebras that preserves the summands in that order. In particular, $\psi$ is an isomorphism
   of $\big(\spi(n,\R)\oplus\spi(1,\R)\big)$-modules.
 \end{lemma}
 \begin{proof}
   Recall that $x_0\in\sS$ satisfies Lemma \ref{lem: H simple}, therefore, there is an algebra isomorphism
   $\psi_0:\spi(n+1,\R)\to\HH$. The inverse image of $\GG_1(x_0)\oplus\GG_2(x_0)$ under the
   isomorphism $\psi_0$ induces the decomposition of $\spi(n+1,\R)$ as a direct sum of
   irreducible $\big(\spi(n,\R)\oplus\spi(1,\R)\big)$-modules, such decomposition satisfies that $\psi_0$
   is an isomorphism of $\big(\spi(n,\R)\oplus\spi(1,\R)\big)$-modules.
 \end{proof}

 Now, let us fix an isomorphism of Lie algebras $\psi:\spi(n+1,\R)\to\HH$ as in Lemma \ref{lem: isomorph}.
 By Lemma \ref{lem: righ-act}, there is an analytic, isometric right
 $\TSp(n+1,\R)$-action on $\TM$. Hence, we can consider the next map:
 \begin{equation*}
   f^\psi:\TSp(n+1,\R)\to\TM, \qquad g\mapsto x_0\cdot g
 \end{equation*}
 which satisfies $df^\psi_{e}(X)=X^*_{x_0}=\psi(X)_{x_0}$ for every $X\in\spi(n+1,\R)$.
 By properties of the map $\psi$ and Lemma \ref{lem: H as spn} we have that $df^\psi_{e}$ is an
 isomorphism that maps $\spi(n,\R)\oplus\spi(1,\R)$ onto $T_{x_0}\FF$ and $\R^{2n,2n}$ onto
 $T_{x_0}\FF^\perp$. The analyticity local diffeomorphism of $f^\psi$ follows
 of the $\TSp(n+1,\R)$-equivariance on its domain.

 With a similar analysis to Lemma 3.2 in \cite{OQ} we have our following result.

 \begin{lemma}\label{lem: metric changed}
   Let $\widehat{g}$ be the metric on $\spi(n+1,\R)$ defined as the pullback under $df^\psi$ of the metric
   $g_{x_0}$ on $T_{x_0}\TM$. Then, $\widehat{g}$ is $\big(\spi(n,\R)\oplus\spi(1,\R)\big)$-invariant.
 \end{lemma}

 By the previous Lemma and the results in Lemma \ref{lem: metric inv} we can rescale the metric along
 the bundles $T\FF_1$, $T\FF_2$ and $T\FF^\perp$ in $M$ such that the new metric $\widetilde{g}$ on
 $\TM$ satisfies $(df^\psi)^*(\widetilde{g}_{x_0})=\textsf{K}_{n+1}$, the Killing form of $\spi(n+1,\R)$.

 Since the elements in $\HH$ preserve the decomposition $TM$ as its direct sum
 $TM=T\FF_1\oplus T\FF_2\oplus T\FF^\perp$ then $\HH\subset\Kill(\TM,\widetilde{g})$. Hence, the elements
 of $\HH$ are Killing vector fields for the metric $\widetilde{g}$, therefore $\widetilde{g}$ is
 invariant under the right $\TSp(n+1,\R)$-action.

 In a similar way we can observe that the isometric $\big(\TSp(n,\R)\times\TSp(1,\R)\big)$-action on
 $\TM$ preserves our rescaled metric $\widetilde{g}$. We also note that the metric $\widetilde{g}$
 is the lift of a correspondingly metric $\widetilde{g}$ in $M$.

 Considering the bi-invariant metric on $\TSp(n+1,\R)$ induced by the Killing form $\textsf{K}_{n+1}$,
 which we denote with the same symbol. The previous paragraphs show that the local diffeomorphism
 $f^\psi:(\TSp(n+1,\R),\textsf{K})\to(\TM,\widetilde{g})$ is a local isometry. Therefore, we have that
 $f^\psi$ is an \emph{isometry}, such result follows from Corollary 29 in \cite[p. 202]{Oneill}, the
 simply connectedness of $\TM$ and the completeness of $(\TSp(n+1,\R),\textsf{K})$.

 Hence, by the previous remarks, we have the following result.
 \begin{lemma}\label{lem: almost done}
   Let $M$ and $G=\big(\TSp(n,\R)\times\TSp(1,\R)\big)$ as in Theorem A, then there is an
   analytic diffeomorphism $f:\TSp(n+1,\R)\to\TM$ and an analytic isometric right $\TSp(n+1,\R)$ on
   $\TM$ such that:
   \begin{itemize}
     \item[$(1)$] On $\TM$, the right $\TSp(n+1,\R)$-action and the left
      $\big(\TSp(n,\R)\times\TSp(1,\R)\big)$-action  commute with each other;
     \item[$(2)$] the map $f$ is $\TSp(n+1,\R)$-equivariant for the natural right $\TSp(n+1,\R)$-action
      on $\TSp(n+1,\R)$;
     \item[$(3)$] with the metric $\widetilde{g}$, obtained by rescaling the original metric ($g$ on $M$) on the
      summands of the direct decomposition $T\TM=T\FF_1\oplus T\FF_2\oplus T\FF^\perp$, the map
      $f:\big(\TSp(n+1,\R),\textsf{K}\big)\to(\TM,\widetilde{g})$ is an isometry, where
      $\textsf{K}$ is the metric on $\TSp(n+1,\R)$ induced by the Killing form of its Lie algebra.
   \end{itemize}
 \end{lemma}

 First, by the results in \cite{Muller} we have that $\Iso(\TSp(n+1,\R))$ has finite many components and
 that $\Iso_0(\TSp(n+1,\R))=L(\TSp(n+1,\R))R(\TSp(n+1,\R))$, where $L(g)$ (resp. $R(g)$) is the left
 (resp. right) translation map on $\TSp(n+1,\R)$ by the element $g\in\TSp(n+1,\R)$.

 Let $\varrho:\TSp(n,\R)\times\TSp(1,\R)\to\Iso_0(\TSp(n+1,\R))$ be the homomorphism generated by the left
 action of $\big(\TSp(n,\R)\times\TSp(1,\R)\big)$ on $\TSp(n+1,\R)$. By the previous paragraph, there are
 two homomorphism $\varrho_L,\varrho_R:\TSp(n,\R)\times\TSp(1,\R)\to\TSp(n+1,\R)$ such that
 $\varrho(h)=L(\varrho_L(h))R(\varrho_R(h))$ for every $h\in\big(\TSp(n,\R)\times\TSp(1,\R)\big)$.

 Since the right $\TSp(n+1,\R)$-action and the left action of $\big(\TSp(n,\R)\times\TSp(1,\R)\big)$ on
 $\TSp(n+1,\R)$ commute each other then $\varrho_R\big(\TSp(n,\R)\times\TSp(1,\R)\big)$ lies in the center
 of $\TSp(n+1,\R)$, such property and the fact that $\big(\TSp(n,\R)\times\TSp(1,\R)\big)$ is connected
 imply that $\varrho_R\big(\TSp(n,\R)\times\TSp(1,\R)\big)=e$. Therefore $\varrho=L\circ\varrho_L$.

 On the other hand, by Lemma \ref{lem: almost done} and the previous paragraphs we have that the subgroup
 $\Sigma_0=\pi_1(M)\cap\Iso_0(\TSp(n+1,\R))$ (since $\pi_1(M)\subset\Iso(\TSp(n+1,\R))$) has finite index in
 $\pi_1(M)$. Considering that the action of
 $\big(\TSp(n,\R)\times\TSp(1,\R)\big)$ on $\TSp(n+1,\R)$ is the lift of an action on $M$ we have that the
 elements in $\Sigma_0$ commute with the elements of $\varrho\big(\TSp(n,\R)\times\TSp(1,\R)\big)$.
 Therefore, since $\sigma_0=L(\sigma_1)R(\sigma_2)$ for some $\sigma_1,\sigma_2\in\Sigma$ then
 $L(\sigma_1)\circ L(\varrho_L(h))=L(\varrho_L(h))\circ L(\sigma_1)$ for every
 $h\in\big(\TSp(n,\R)\times\TSp(1,\R)\big)$, hence
 $\Sigma_0\subset L(Z_{\TSp(n+1,\R)}(\TSp(n,\R)\times\TSp(1,\R)))R(\TSp(n,\R)\times\TSp(1,\R))$.

 By the results in Lemma \ref{lem: centralizer} we have that $R(\TSp(n,\R)\times\TSp(1,\R))$ has finite
 index in $L(Z_{\TSp(n+1,\R)}(\TSp(n,\R)\times\TSp(1,\R)))R(\TSp(n,\R)\times\TSp(1,\R))$. In particular,
 $\Sigma=\Sigma_0\cap R(\TSp(n,\R)\times\TSp(1,\R))$ is a finite index subgroup of
 $\Sigma_0$, and therefore it has finite index in $\pi_1(M)$.

 The natural identification of $R(\TSp(n,\R)\times\TSp(1,\R))$ with $\TSp(n,\R)\times\TSp(1,\R)$
 induces to consider $\Sigma$ as a discrete subgroup of $\TSp(n,\R)\times\TSp(1,\R)$ such
 that the quotient map $\TSp(n+1,\R)/\Sigma$ is a finite covering map of the manifold $M$. Let
 $\xi:\TSp(n+1,\R)/\sigma\to M$ be the finite covering map, previously defined, for the
 left action of $\TSp(n,\R)\times\TSp(1,\R)$ on $\TSp(n+1,\R)/\Sigma$ given by the homomorphism
 $\varrho_L:\TSp(n,\R)\times\TSp(1,\R)\to\TSp(n+1,\R)$, we have that the map $\xi$ is
 $\TSp(n,\R)\times\TSp(1,\R)$-equivariant. We also observe that $\xi$ is an isometry for the metric
 $\widehat{g}$, as it is defined in Lemma \ref{lem: metric changed}.

 Finally, in order to complete the proof of Theorem A we only need to prove that the subgroup
 $\Sigma$ is a lattice in $\TSp(n+1,\R)$. Such result is shown in the following lemma which proof
 is similar to Lemma 3.4 in \cite{OQ}.

 \begin{lemma}
   Let $\mathrm{vol}_{g}$ and $\mathrm{vol}_{\widehat{g}}$ define the volume elements on $M$, for the
   original metric and the rescaled metric $\widehat{g}$, in Lemma \ref{lem: metric changed}, respectively.
   Then, there is a constant $C_{\widehat{g}}>0$ such that
   $\mathrm{vol}_{\widehat{g}}=C_{\widehat{g}}\mathrm{vol}_g$.
 \end{lemma}

\appendix\section{Modules and representations}
 We start this appendix with the following result about decomposition into irreducible modules of
 non-compact simple Lie groups
 \begin{lemma}\label{lem: A rep}
   Let $G\subsetneq H$ be non-compact simple Lie groups and $(\pi,V)$ an irreducible representation of $H$ such that
   \begin{equation*}
     \pi|_{G}\simeq W\oplus \bigoplus_j W_j
   \end{equation*}
   is its direct sum decomposition into irreducible $G$-modules such that $W$ has multiplicity $1$ in $V$.
   If $\widetilde{W}\subset V$ is a $G$-invariant irreducible representation with $\widetilde{W}\simeq W$
   then for every $g\in G$ we have that $\pi(g)(\widetilde{W})=W$.
 \end{lemma}

 The previous lemma will be used to understand the inclusion of $\spi(n,\R)$ in $\so(2n,2n)$, for $n\geq3$.
 In search of such understanding we recall the following:
 \begin{equation}\label{eq: sp(n,R)}
   \spi(n,\R)=\big\{M\in \gl_{2n}(\R)\big|M^T\widetilde J+\widetilde JM=0\big\},
 \end{equation}
 where
 \begin{equation*}
   \widetilde J=\begin{bmatrix}
       0 & I_{n} \\
       -I_{n} & 0
     \end{bmatrix}.
 \end{equation*}
 In a more explicitly way, we have
 \begin{equation*}
   \spi(n,\R)=\bigg\{\begin{bmatrix}
                      A & B \\
                      C & -A^T
                    \end{bmatrix}\in\gl_{2n}(\R)\bigg|A\in\gl_{n}(\R), B^T=B, C^T=C\bigg\}.
 \end{equation*}
 Note, by Table II of \cite{Berger}, that
 $(\sli(2n,\R),\spi(n,\R))$ is a symmetric pair. Even more $\sli(2n,\R)=\spi(n,\R)\oplus\pi_2$
 where $\pi_2$ is the irreducible representation of $\spi(n,\R)$ corresponding to its second highest
 weight $\varpi_2$ (see Theorem 5.5.15 and its immediate consequences in \cite{Goodman}).

 On the other hand, we have that
 \begin{equation}\label{eq: so(2n,2n)}
   \so(2n,2n)=\big\{M\in\gl_{4n}(\R)\big|M^TJ+JM=0\big\},
 \end{equation}
 where
 \begin{equation*}
   J=\begin{bmatrix}
       0 & I_{2n} \\
       I_{2n} & 0
     \end{bmatrix}.
 \end{equation*}
 That is
 \begin{equation*}
   \so(2n,2n)=\bigg\{\begin{bmatrix}
                      A_0 & B_0 \\
                      C_0 & -A_0^T
                    \end{bmatrix}\in\gl_{4n}(\R)\bigg|A_0\in\gl_{2n}(\R), B_0,C_0\in\so(n,\R)\bigg\}.
 \end{equation*}

 Since $(\so(2n,2n),\sli(2n,\R)\oplus\R)$ is also a symmetric pair (see \cite[Table II]{Berger}) then we have a
 guarantee of an inclusion of $\spi(n,\R)$ into $\so(2n,2n)$. Therefore,
 an inclusion of the Lie algebra $\spi(n,\R)$ into the Lie algebra $\so(2n,2n)$ is given
 in the following way
 \begin{eqnarray*}
   \spi(n,\R)&\hookrightarrow&\so(2n,2n)\\
   \begin{bmatrix}
                      A & B \\
                      C & -A^T
                    \end{bmatrix}&\mapsto&
   \begin{bmatrix}
     A & B & 0 & 0 \\
     C & -A^T & 0 & 0 \\
     0 & 0 & -A^T & -C \\
     0 & 0 & -B & A
   \end{bmatrix}.
 \end{eqnarray*}
 Recall that $\gl_{2n}(\R)=\sli(2n,\R)\oplus\R$ then, by above, we have that
 $\gl(2n,\R)=\spi(n,\R)\oplus\pi_2\oplus\R$. On the other hand, by Table II in \cite{Berger}, we have that
 \begin{equation*}
   \so(2n,2n)=\sli(2n,\R)\oplus\R\oplus\pi^1_2(\sli(2n,\R))\oplus\pi^2_2(\sli(2n,\R))
 \end{equation*}
 where $\pi^i_2(\sli(2n,\R))$ denotes the irreducible representation of $\sli(n,\R)$ corresponding
 to its second highest weight, for $i=1,2$. Because $\pi^i_2(\sli(2n,\R))=\pi_2\oplus\R$, its
 decomposition as a direct sum of irreducible $\spi(n,\R)$-modules, then
 \begin{equation*}
   \so(2n,2n)\simeq\spi(n,\R)\oplus\bigoplus_{i=1}^3\pi^i_2\oplus\bigoplus_{i=1}^3\pi^i_0,
 \end{equation*}
 as a direct sum of irreducible $\spi(n,\R)$-modules where $\pi^i_0:=\R$ is the trivial representation
 of $\spi(n,\R)$ corresponding to its highest weight $\varpi_0$, for $i=1,2,3$.

 Let $W_0\in\so(2n,2n)$ be an  element which commutes with every element of the previous inclusion, of $\spi(n,\R)$
 into $\so(2n,2n)$, then taking particular elements in $\spi(n,\R)$, it can be proven that
 \begin{equation*}
   W_0=\begin{bmatrix}
     aI_n &  & 0 & bI_n \\
     0 & aI_n & -bI_n & 0 \\
     0 & -cI_n & -aI_n & 0 \\
     cI_n & 0 & 0 & -aI_n
   \end{bmatrix}
 \end{equation*}
 for some $a,b,c\in\R$. Therefore, with the above inclusion of $\spi(n,\R)$ into $\so(2n,2n)$ we have
 that an inclusion of $\spi(n,\R)\oplus\spi(1,\R)$ into $\so(2n,2n)$ is given as follow
 \begin{eqnarray*}
   \spi(n,\R)\oplus\spi(1,\R)&\hookrightarrow&\so(2n,2n)\\
   \begin{bmatrix}
                      A & B \\
                      C & -A^T
                    \end{bmatrix}+
   \begin{bmatrix}
                      a & b \\
                      c & -a
                    \end{bmatrix}
                    &\mapsto&
   \begin{bmatrix}
     A+aI_n & B & 0 & bI_n \\
     C & -A^T+aI_n & -bI_n & 0 \\
     0 & -cI_n & -A^T-aI_n & -C \\
     cI_n & 0 & -B & A-aI_n
   \end{bmatrix}.
 \end{eqnarray*}

 \begin{lemma}\label{lem: inclusion}
   There is, up to isomorphism, an unique inclusion of $\spi(n,\R)\oplus\spi(1,\R)$ into $\so(2n,2n)$.
 \end{lemma}
 \begin{proof}
   Since $(\sli(2n,\R),\spi(n,\R))$ and $(\so(2n,2n),\sli(2n,\R)\oplus\R)$ are symmetric pairs we have
   an inclusion of $\spi(n,\R)$ into $\so(2n,2n)$. By the previous paragraphs we have that such inclusion
   induces a decomposition of $\so(2n,2n)$ into a direct sum of irreducible $\spi(n,\R)$-modules as follow
   \begin{equation*}
   \so(2n,2n)\simeq\spi(n,\R)\oplus\bigoplus_{i=1}^3\pi^i_2\oplus\bigoplus_{i=1}^3\pi^i_0.
 \end{equation*}

 Lemma \ref{lem: A rep} shows that the inclusion of $\spi(n,\R)$ into $\so(2n,2n)$, is unique up to
 isomorphism. On the other hand, by the simplicity of $\spi(1,\R)$ and since the inclusion of $\spi(1,\R)$
 into $\so(2n,2n)$ is contained in $Z_{\so(2n,2n)}(\spi(n,\R))$ (the centralizer of $\spi(n,\R)$ in
 $\so(2n,2n)$) then we have that the inclusion of $\spi(n,\R)\oplus\spi(1,\R)$ into $\so(2n,2n)$ is unique
 up to isomorphism.
 \end{proof}

 An immediate consequence of the previous lemma is the following corollary.

 \begin{corollary}\label{cor: inclusion}
   With the above inclusion of $\spi(n,\R)\oplus\spi(1,\R)$ into $\so(2n,2n)$, given in Lemma \ref{lem: inclusion},
   we have that $\R^{2n,2n}$ is an irreducible $\big(\spi(n,\R)\oplus\spi(1,\R)\big)$-module.
 \end{corollary}

 Next, we analyze the representations of $\spi(n,\R)$ through the study of their correspondent complexification,
 all these facts can be found in \cite{Onishchik}.

 Let $\g_0$ be an real Lie algebra and let $\rho:\g_0\to\gl(V_0)$ be a representation of $\g_0$ in a real vector
 space $V_0$. Let us denote $V=V_0(\C)$ and $\g=\g_0(\C)$. Here, we have two complexification operations related to
 $\rho$. First, we have a complex representation $\rho^\C:\g_0\to\gl(V)$, obtained extending any $\rho(x)$,
 $x\in\g_0$, to a complex linear operator in $V$. Second, we can extend $\rho^\C$ to a homomorphism of complex
 Lie algebras $\rho(\C):\g\to\gl(V)$.

 The following result uses the previous complex representations and it gives a classification of irreducible
 real representations.

 \begin{theorem}[{\cite[Th 1, Sect. 8]{Onishchik}}]\label{th: complexification}
   Any irreducible real representation $\rho:\g_0\to\gl(V_0)$ of a real Lie algebra $\g_0$ satisfies precisely
   on of the following two conditions:
   \begin{enumerate}
     \item[(i)] $\rho^\C$ is an irreducible complex representation;
     \item[$(ii)$] $\rho=\rho'_\R$, where $\rho'$ is an irreducible complex representation admitting no
      invariant real structures.
   \end{enumerate}
   Conversely, any real representation $\rho$ satisfying $(i)$ or (ii) is irreducible.
 \end{theorem}

 Let $\rho:\g_0\to\gl(V)$ be a self-conjugate irreducible complex representation. The \emph{Cartan Index}
 of $\rho$ is  $\varepsilon(\rho)=\mathrm{sgn}({c})=\pm1$, where $c$ is defined by the following condition:
 $S^2=ce$, where $S$ is an automorphism of $V$ commuting with $\rho$. By the results in Section 8 in
 \cite{Onishchik}, we have that an irreducible
 complex representation $\rho:\g\to\gl(V)$ admits an invariant real structure if and only if $\rho$ is
 self-conjugate and its Cartan index is equal to $1$.

 By Theorem 3 in Section 8 and Table 5 in \cite{Onishchik} we have that the irreducible complex representations
 of $\spi(n,\R)$ are self-conjugate and their Cartan index is always $1$. Therefore, the study of real
 irreducible representations of $\spi(n,\R)$ is similar to the study of irreducible complex representations
 of $\spi(n,\C)$.

 From Section 5.5.2 in \cite{Goodman} we have a bijection between (finite) complex representations of
 a complex semisimple Lie algebra $\g$ and the set of dominant integral weights associated to $\g$.
 The dominant integral
 weights are of the form $n_1\varpi_1+n_2\varpi_2+\cdots+n_k\varpi_k$ with $n_k\in\mathbb{N}$, where
 $\varpi_1,\varpi_2,\ldots,\varpi_k$ are the fundamental weights of $\g$.

 It is clear that the dimension of the representation associated to
 $n_1\varpi_1+n_2\varpi_2+\cdots+n_k\varpi_k$  is bigger or equal to the dimension of the representation
 associated to $n_j\varpi_j$ and this, if $n_j\neq0$, to the dimension of the representation associated
 to $\varpi_j$, for every $j\in\{1,2,\ldots k\}$.

 In our case, $\g=\spi(n,\C)$, and therefore for $\g_0=\spi(n,\R)$ we have that the fundamental weights
 are $\varpi_1,\varpi_2,\ldots,\varpi_n$.

 \begin{lemma}\label{lem: dimension 4n}
   The dimension of the representation of $\spi(n,\R)$ associated to $\varpi_j$ is bigger that $4n$ when
   $n\geq3$ and $2\leq j\leq n$.
 \end{lemma}
 \begin{proof}
   By Corollary 5.5.17 in \cite{Goodman} we have that the dimension of the complex representation of
   $\spi(n,\C)$ associated to $\varpi_j$, and hence the real representation of $\spi(n,\R)$ associated to
   $\varpi_j$, is $\binom{2n}{j}-\binom{2n}{j-2}$ (with the convention that $\binom{m}{p}=0$ when $p$
   is a negative integer).

   If $j=2$ then
   \begin{equation*}
     \binom{2n}{2}-\binom{2n}{0}=n(2n-1)-1,
   \end{equation*}
   which satisfies that
   \begin{equation*}
     n(2n-1)-1>4n \iff j\geq3.
   \end{equation*}

   If $j=3$ then
   \begin{equation*}
     \binom{2n}{3}-\binom{2n}{1}=\frac{2n(2n-1)(n-1)}{3}-2n=\frac{2n(2n^2-3n-2)}{3}
   \end{equation*}
   satisfying that
   \begin{eqnarray*}
     \frac{2n(2n^2-3n-2)}{3}>4n &\iff& \frac{2n(2n^2-3n-8)}{3}>0 \\
      &\iff& 2n(2n^2-3n-8)>0 \\
      &\iff& 2n^2-3n-8>0 \\
      &\iff& n\geq3.
   \end{eqnarray*}

   For $4\leq k\leq n$ we have
   \begin{eqnarray*}
     \binom{2n}{k}-\binom{2n}{k-2} &=& \binom{2n}{k-2}\Bigg(\frac{(2n-k+3)(2n-k+4)}{(k-1)k}-1\Bigg) \\
      &\geq& \binom{2n}{k-2}\Bigg(\frac{(2k-k+3)(2k-k+4)}{(k-1)k}-1\Bigg) \\
      &=& \binom{2n}{k-2}\Bigg(\frac{(k+3)(k+4)}{(k-1)k}-1\Bigg) \\
      &=& \binom{2n}{k-2}\frac{(k+3)(k+4)-k(k-1)}{(k-1)k} \\
      &=& \binom{2n}{k-2}\frac{8k+12}{(k-1)k}\\
      &=& \binom{2n}{k-2}\frac{8(k-1)+20}{(k-1)k}\\
      &=& \binom{2n}{k-2}\Bigg(\frac{8}{k}+\frac{20}{(k-1)k}\Bigg)\\
      &>& \binom{2n}{k-2}\frac{8}{k}\\
      &\geq& \binom{2n}{k-2}\frac{8}{n}\\
      &\geq& \binom{2n}{2}\frac{8}{n}\\
      &=&8(2n-1),
   \end{eqnarray*}
   here
   \begin{equation*}
     8(2n-1)>4n \iff n>\frac{2}{3} \quad \text{in particular if} \quad  n\geq4.
   \end{equation*}
 \end{proof}

 \begin{remark}\label{rem: varpi1}
   First, by the definition of $\spi(n,\R)$, we have that the representation of $\spi(n,\R)$ on $\R^{2n}$,
   corresponding to its highest weight $\varpi_1$, preserves a non-degenerate \emph{skew-symmetric} bilinear
   form. Hence, such representation cannot preserve a non-de\-ge\-ne\-ra\-te \emph{symmetric} bilinear form.
   On the other hand, recall that the representation of $\spi(n,\R)$ on $\R$ corresponds to the trivial
   homomorphism.
 \end{remark}

 With the observations in Remark \ref{rem: varpi1} and Lemma \ref{lem: dimension 4n} we can now
 show a representation of $\spi(n,\R)$ with the minimal dimension preserving a non-degenerate symmetric
 bilinear form.

 \begin{lemma}\label{lem: dim rep}
   The minimal dimension of a non-trivial representation of $\spi(n,\R)$ which preserves a non-degenerate symmetric
   bilinear form is $4n$, even more, such representation is isomorphic, as $\spi(n,\R)$-module, to
   $\R^{2n}\oplus\R^{2n}$.
 \end{lemma}
 \begin{proof}
   Let $V$ be a non-trivial representation of $\spi(n,\R)$ which preserves a non-degenerate symmetric
   bilinear form. By Lemma \ref{lem: dimension 4n} and Remark \ref{rem: varpi1} we have that $\dim(V)=r\geq2n$.

   If $\dim(V)=2n$ then, by Lemma \ref{lem: dimension 4n} and Remark \ref{rem: varpi1}, we have that
   $V\simeq\R^{2n}$ (as $\spi(n,\R)$-module) which can not preserve a non-degenerate symmetric bilinear form.

   If $2n+1\leq\dim(V)\leq4n-1$ then, by Lemma \ref{lem: dimension 4n} and Remark \ref{rem: varpi1}, we have
   that $V\simeq\R^{2n}\oplus\bigoplus_{j=1}^{r-2n}\R$. Let $\LA\cdot,\cdot\RA$ be the non-degenerate
   symmetric bilinear form on $\R^{2n}\oplus\bigoplus_{j=1}^{r-2n}\R$ induced by its homomorphism with $V$,
   note that $\LA,\RA$ is preserved by the action of $\spi(n,\R)$. Here, if $x\in\R^{2n}$ and
   $h\in\oplus_{j=1}^{r-2n}\R$ we have that for every $A\in\spi(n,\R)$
   \begin{equation*}
     0 = \LA A\cdot x,h\RA+\LA x,A\cdot h\RA
       = \LA A\cdot x,h\RA+\LA x,0\RA
       = \LA A\cdot x,h\RA.
   \end{equation*}
   Because the elements have been taken arbitrarily, that implies that $\LA\cdot,\cdot\RA$ is non-degenerated
   when is restricted to $\R^{2n}$, which is not possible. Thence, resuming $\dim(V)\geq4n$.

   Since $\big(\sli(2n,\R),\spi(n,\R)\big)$ and $\big(\so(2n,2n),\sli(2n,\R)\oplus\R\big)$ are symmetric
   pairs then there is a non-trivial representation of $\spi(n,\R)$, with dimension $4n$,
   preserving a non-degenerate symmetric bilinear form.

   If $\dim(V)=4n$ then, by Lemma \ref{lem: dimension 4n}, $V$ must be isomorphic to $\R^{2n}\oplus\R^{2n}$,
   $\R^{2n}\oplus\bigoplus_{j=1}^{2n}\R$ or well to $\bigoplus_{j=1}^{4n}\R$. The last two options are not
   possible as is shown previously. Therefore we have that $V$ is isomorphic to $\R^{2n}\oplus\R^{2n}$.
 \end{proof}

 And as consequence of the previous result and Corollary \ref{cor: inclusion} we have the next lemma.

 \begin{lemma}\label{lem: n+1 as n+1}
   The decomposition of $\spi(n+1,\R)$ as a direct sum of irreducible
   $\big(\spi(n,\R)\oplus\spi(1,\R)\big)$-modules is given as
   $\spi(n+1,\R)=\spi(n,\R)\oplus\spi(1,\R)\oplus\R^{2n,2n}$.

 \end{lemma}
 \begin{proof}
   Recall, by \cite{Berger}, that $(\spi(n+1,\R),\spi(n,\R)\oplus\spi(1,\R))$ is a symmetric pair. On
   the other hand, since any Cartan involution on $\spi(n,\R)\oplus\spi(1,\R)$ can be extended to a
   Cartan involution on $\spi(n+1,\R)$ and all Cartan involution are conjugates we have then that the
   complement of $\spi(n,\R)\oplus\spi(1,\R)$ in $\spi(n+1,\R)$ is a non-degenerated vector subspace
   with dimension $4n$ which is a non-trivial $\big(\spi(n,\R)\oplus\spi(1,\R)\big)$-module.

   By Corollary \ref{cor: inclusion} and Lemma \ref{lem: dim rep} we have that such complement is isomorphic
   to $\R^{2n,2n}$ as a $\big(\spi(n,\R)\oplus\spi(1,\R)\big)$-module, which is irreducible. Therefore,
   we have our desired decomposition.
 \end{proof}

 As a direct consequence of the previous lemma we know the centralizer of
 $\big(\spi(n,\R)\oplus\spi(1,\R)\big)$ in $\spi(n,\R)$, such result can be found in the next lemma.

 \begin{lemma}\label{lem: centralizer}
   Assume that $\rho:\TSp(n,\R)\times\TSp(1,\R)\to\TSp(n+1,\R)$ is an homomorphism of Lie groups which is
   an immersion. Then, %% the centralizer
   $Z_{\TSp(n+1,\R)}\rho\big(\TSp(n,\R)\times\TSp(1,\R)\big)$, the centralizer  of
   $\TSp(n,\R)\times\TSp(1,\R)$ in $\TSp(n+1,\R)$ contains the center of $\TSp(n+1,\R)$ as a finite index
   subgroup.
 \end{lemma}
 \begin{proof}
   As a consequence of the decomposition of $\spi(n+1,\R)$ as a direct sum of irreducible
   $\big(\spi(n,\R)\oplus\spi(1,\R)\big)$-modules, shown in Lemma \ref{lem: n+1 as n+1}, we have that
   $\mathfrak{z}_{\spi(n+1,\R)}\big(\spi(n,\R)\oplus\spi(1,\R)\big)=0$, therefore
   $Z_{\TSp(n+1,\R)}\rho\big(\TSp(n,\R)\times\TSp(1,\R)\big)$ is discrete.

   By Lemma $1.1.3.7$ in \cite{Warner} we have that $Z_{\TSp(n+1,\R)}\rho\big(\TSp(n,\R)\times\TSp(1,\R)\big)$
   is finite. Since $Z(\TSp(n+1,\R))\subseteq Z_{\TSp(n+1,\R)}\rho\big(\TSp(n,\R)\times\TSp(1,\R)\big)$,
   therefore we have our result.
 \end{proof}

 By the results in Lemma \ref{lem: dim rep} we have that $\R^{2n,2n}$ is isomorphic to $\R^{2n}\oplus\R^{2n}$
 as $\spi(n,\R)$-module. On the other hand, by the inclusion of $\spi(n,\R)$ into $\so(2n,2n)$ (unique up
 to isomorphisms) as in Lemma \ref{lem: inclusion}, and the remarks previous to such lemma, we have that
 the vector subspaces $\R^{2n}$ belongs to the nullcone.

 Next, we will see the properties of $\big(\spi(n,\R)\oplus\spi(1,\R)\big)$-invariants inner products on
 $\spi(n,\R)$, $\spi(n,\R)$ and $\R^{2n,2n}$.

 \begin{lemma}\label{lem: metric inv}
   Let $\LA\cdot,\cdot\RA_n$, $\LA\cdot,\cdot\RA_1$ and $\LA\cdot,\cdot\RA_0$ be inner products on
   $\spi(n,\R)$, $\spi(1,\R)$ and $\R^{2n,2n}$, respectively. Assume that $\LA\cdot,\cdot\RA_n$ is
   $\spi(n,\R)$-invariant, $\LA\cdot,\cdot\RA_1$ is $\spi(1,\R)$-invariant and $\LA\cdot,\cdot\RA_0$ is
   $\big(\spi(n,\R)\oplus\spi(1,\R)\big)$-invariant, for $n\geq3$. Then there exist $a_0,a_1,a_n\in\R$
   such that $a_0\LA\cdot,\cdot\RA_0+a_1\LA\cdot,\cdot\RA_1+a_n\LA\cdot,\cdot\RA_n$
   is the Killing form of $\spi(n+1,\R)$.
 \end{lemma}
 \begin{proof}
   The proof follows from Schur's Lemma, the irreducibility of $\R^{2n,2n}$ as a
   $\big(\spi(n,\R)\oplus\spi(1,\R)\big)$-module and the uniqueness of the Killing form of complex
   simple Lie algebras.
 \end{proof}

\end{document}